\documentclass[reqno,11pt]{amsart}
\usepackage[T1]{fontenc}
\usepackage[utf8]{inputenc}
\usepackage{graphicx, amsmath, amssymb, amscd, mathtools, hyperref, amsthm, euscript, 
amsfonts,bm,color}  
\usepackage{mathrsfs}  
\DeclareMathAlphabet{\mathpzc}{OT1}{pzc}{m}{it}
\usepackage{bbm}
\usepackage[center]{caption}

\newtheorem{thm}[equation]{Theorem}

\newtheorem*{theorem*}{Theorem}
\newtheorem*{cor*}{Corollary}
\newtheorem*{rmk*}{Remark}
\newtheorem*{prop*}{Proposition}

\newtheorem{rmk}[equation]{Remark}\newtheorem{remark}[equation]{Remark}

\newtheorem{prop}[equation]{Proposition}

\newtheorem{cor}[equation]{Corollary}
\newtheorem{lem}[equation]{Lemma}

\newtheorem{dfn}[equation]{Definition}

\numberwithin{equation}{section}

\numberwithin{equation}{section}

\newcommand{\e}{\epsilon}

\newcommand{\bms}{\mathsf{m}}

\newcommand{{\grinv}}{{\Cal G}^{-r}}

\newcommand{\ba}{\backslash}\newcommand{\bs}{\backslash}

\newcommand{\al}{\alpha}
\newcommand{\be}{\beta}

\newcommand{\Cal}{\mathcal}

\newcommand{\bp}{\begin{pmatrix}}
\newcommand{\ep}{\end{pmatrix}}

\renewcommand{\bp}{{\rm bp}}

\newcommand{\la}{\lambda}

\newcommand{\norm}[1]{\lVert #1 \rVert}

\newcommand{\op}{\operatorname}

\newcommand{\BMS}{\operatorname{BMS}}

\newcommand{\cl}[1]{\overline{#1}}

\renewcommand{\setminus}{-}

\newcommand{\Ad}{\operatorname{Ad}}

\newcommand{\Om}{\Omega}

\newcommand\Hom{{\rm Hom}}

\newcommand{\ga}{\gamma}
\newcommand{\F}{\mathcal F}
\newcommand{\La}{\Lambda}
\renewcommand{\i}{\op{i}}

\def\Im{\operatorname{Im}}
\def\Hom{\operatorname{Hom}}

\def\e{\mathrm{e}}
\def\i{\mathrm{i}}

\newcommand{\Ga}{\Gamma}\newcommand{\bb}{\mathbb}
\newcommand{\cal}{\mathcal}
\renewcommand{\e}{\varepsilon}
\renewcommand{\epsilon}{\e}

\usepackage{amsmath}
\usepackage{array}
\allowdisplaybreaks

\makeatletter
\newcounter{elimination@steps}
\newcolumntype{R}[1]{>{\raggedleft\arraybackslash$}p{#1}<{$}}
\def\elimination@num@rights{}
\def\elimination@num@variables{}
\def\elimination@col@width{}

\newcommand{\eliminationstep}[2]
{
    \ifnum\value{elimination@steps}>0\longrightarrow\quad\fi
    \left[
        \ifnum\elimination@num@rights>0
            \begin{array}
            {@{}*{\elimination@num@variables}{R{\elimination@col@width}}
            |@{}*{\elimination@num@rights}{R{\elimination@col@width}}}
        \else
            \begin{array}
            {@{}*{\elimination@num@variables}{R{\elimination@col@width}}}
        \fi
            #1
        \end{array}
    \right]
    & 
    \begin{array}{l}
        #2
    \end{array}
    \addtocounter{elimination@steps}{1}
}
\makeatother

\begin{document}

\author{Mikołaj Frączyk}
\address{Department of Mathematics, University of Chicago, 5734 S University Ave, Chicago, IL 60637}
\curraddr{Faculty of Mathematics and Computer Science, Jagiellonian University, ul. Łojasiewicza 6, 30-348 Krak{\'o}w, Poland}
\email{mikolaj.fraczyk@uj.edu.pl}

\author{Minju Lee}
\address{Department of Mathematics, University of Chicago, 5734 S University Ave, Chicago, IL 60637}
\email{minju1@uchicago.edu}

\title[]{Discrete subgroups with finite Bowen-Margulis-Sullivan measure\\
in higher rank
}
\begin{abstract}
    Let $G$ be a connected semisimple real algebraic group and $\Ga<G$ be a Zariski dense discrete subgroup.
    We prove that if $\Ga\ba G$ admits any finite Bowen-Margulis-Sullivan measure, then $\Ga$ is virtually a product of higher rank lattices and discrete subgroups of rank one factors of $G$.
    This may be viewed as a measure-theoretic analogue of classification of convex cocompact actions by Kleiner-Leeb \cite{KL} and Quint \cite{Quint2}, which was conjectured  by Corlette in 1994.
    The key ingredients in our proof are the product structure of leafwise measures and the high entropy method of Einsiedler-Katok-Lindenstrauss \cite{EKL}.
    In a companion paper jointly with Edwards and Oh \cite{EFLO}, we use this result to show that the bottom of the $L^2$ spectrum has no atom in any infinite volume quotient of a higher rank simple algebraic group.
\end{abstract}
\thanks{
M.F. was partially supported by the Dioscuri programme initiated by the Max Planck Society, jointly managed with the National Science Centre in Poland, and mutually funded by Polish the Ministry of Education and Science and the German Federal Ministry of Education and Research.
}
\maketitle
\section{Introduction}

Let $G$ be a higher rank semisimple Lie group. The lattices of $G$ are more or less classified thanks to the Margulis arithmeticity theorem. In contrast, the structure of infinite covolume subgroups remains mysterious, with only a few completely general results available (e.g., see \cite{FG}). Special classes of subgroups, most notably Anosov subgroups, have been investigated in more depth (e.g., see \cite{BCLS, BPS, FockG, Hi, GG, GW, KLP, La} and \cite{Ka, Win} for survey articles). Despite these efforts, our understanding of general {\textquotedblleft}large{\textquotedblright} discrete subgroups of $G$ remains limited. Not many constructions are known, and we still lack the theory that would explain the scarcity of examples.

Quint (\cite{Quint1, Quint3}) has shown that Patterson-Sullivan theory (\cite{Pat, Sul}) can be applied to any discrete subgroup $\Gamma < G$, providing a way to study discrete subgroups of $G$ via their conformal measures and the associated Bowen-Margulis-Sullivan (BMS) measures on the quotient space $\Gamma\bs G$. These measures are invariant under the action of a maximal real split torus of $G$, which, in our setting, has rank at least two. Einsiedler-Katok \cite{EK} and Einsiedler-Katok-Lindenstrauss \cite{EKL} have shown that, under mild positive entropy and certain recurrence assumptions, the actions of higher rank tori on homogeneous spaces exhibit many rigidity phenomena that are absent in the rank one case. In this paper, we show how one can leverage this rigidity to gain a better understanding of discrete subgroups of $G$.

In our main result (Theorem \ref{thm.BMS}) we classify subgroups $\Gamma$ that admit a finite BMS measure, which provides a measured analogue of the theorem of Kleiner-Leeb and Quint classifying higher rank convex cocompact actions. We believe that our approach using entropy rigidity of higher rank actions will yield further results in this area.

\subsection*{Higher rank convex cocompact actions}
Let $X$ be a Riemannian symmetric space of non-compact type.
It is known that the group $G$ consisting of all orientation preserving isometries of $X$ is a connected semisimple real algebraic group\footnote{by this we mean, the connected component of the set of real points of a semisimple algebraic group defined over $\bb R$.}.
Let $\Ga<G$ be a Zariski dense discrete subgroup.
We say $\Ga$ is \textit{convex cocompact} if there exists a non-empty $\Ga$-invariant closed convex subset $C\subset X$ on which $\Ga$ acts cocompactly.

A conjecture of Corlette from 1994, now resolved by the work of Kleiner-Leeb \cite{KL}, says that the only convex cocompact actions in higher rank symmetric spaces are those arising from the rank one constructions. 
More precisely, we have the following:
\begin{thm}[Kleiner-Leeb]\label{thm.KL}
    If $C\subset X$ is a non-empty $\Ga$-invariant closed convex subset on which $\Ga$ acts cocompactly, then $C=C_1\times X_2$ for some non-empty closed subset $C_1\subset X_1$ where $X_1$ $($resp. $X_2)$ denotes the product of rank one $($resp. rank $\geq 2)$ factors of $X$, so that $X=X_1\times X_2$.
\end{thm}
\begin{rmk}
\normalfont
The original article deals with a more general setup from which Theorem \ref{thm.KL} follows as a special case. The more general version takes into account additional Euclidean factor for $X$, and also provides more information about the structure of $C_1$.
We refer the readers to \cite[Thm. 1.1]{KL} for details.
\end{rmk}
In particular, if there exists a convex cocompact subgroup of $G$ which is not a cocompact lattice, $G$ necessarily has a rank one factor.

On the other hand, independently in \cite{Quint2}, Quint established an obstruction to a higher rank convex cocompact action using a different definition.
We give a precise statement of his result since it is quite close to statement we are going to prove.
Let $G$ and $\Ga$ be as before, $A$ be a maximal real split torus of $G$, $Z$ be the centralizer of $A$, and $M$ be the maximal compact subgroup of $Z$ so that $Z=AM$.

Let $\Om_\Ga$ denote the smallest closed subset of $\Ga\ba G/M$ containing all periodic orbits of one-parameter subgroups of $A$.
When $G$ is of rank one, the $A$ action corresponds to the geodesic flow and $\Om_\Ga$ is precisely the non-wandering set of the geodesic flow.
We have:

\begin{thm}[Quint]\label{thm.Q}
    If $\Om_{\Ga}$ is compact, then there exists connected subgroups $G_0,G_1,\cdots,G_q$ with corresponding Zariski dense discrete subgroups $\Ga_0,\Ga_1,\cdots,\Ga_q$ such that the following holds:
\begin{enumerate}
    \item $G=G_0G_1\cdots G_q$ and $\Ga$ is commensurable
    \footnote{For $\Ga$, $\Ga'<G$, we say $\Ga$ is commensurable with $\Ga'$ if $\Ga\cap\Ga'$ is a subgroup of finite index in both $\Ga$ and $\Ga'$.}
    with $\Ga_0\Ga_1\cdots\Ga_q$.
    \item $\Ga_0$ is a cocompact lattice in $G_0$.
    \item For all $1\leq i\leq q$, $G_i$ is a simple Lie group of rank one, and $\Ga_i$ is a convex cocompact subgroup of $G_i$.
\end{enumerate}
\end{thm}
It is known that 
Theorem \ref{thm.Q} can be used to deduce Theorem \ref{thm.KL} \cite[Thm. 5.1]{Quint2}.

In this paper, we obtain a measure-theoretic analogue of Theorem \ref{thm.Q}.
We are concerned with the question {\textquotedblleft}which discrete subgroup $\Ga$ of a semisimple real algebraic group $G$ can admit a finite 
\textit{geometric measure} on $\Ga\ba G${\,\textquotedblright}.

\subsection*{Bowen-Margulis-Sullivan measures}
Bowen-Margulis-Sullivan measure was first introduced by Margulis \cite{Mar} in his 1970 thesis, for the study of Anosov flows on compact manifolds.
Since then it became an important object in the field of ergodic theory and dynamics, and it had been studied extensively in various different contexts.
For a general semisimple Lie group, it can be defined as follows:

Let $G$, $A$, $M$ and $\Ga$ be as before.
Choose a closed Weyl chamber $A^+\subset A$, and denote by $\op{int}A^+$ its interior.
The maximal horospherical subgroups associated to the choice of $A^+$ are given by
$$
N^\pm=\{g\in G : a^k g a^{-k}\to e\text{ as }k\to\pm\infty\text{ for some $a\in \op{int}A^+$}\}.
$$
Let $P^\pm=MAN^\pm$ be the corresponding minimal parabolic subgroups.
For all $x\in\Ga\ba G$, let 
$$
W^{s}(x)=xN^-,\quad W^{cu}(x)=xP^+
$$
be the (global) stable, and central-unstable manifold passing through $x$ corresponding to an element $a\in\op{int}A^+$.
For $\e>0$, let 
$$
W_\e^{s}(x)=xN_\e^-,\quad W_\e^{cu}(x)=xP_\e^+
$$
where $N_\e^-$ and $P_\e^+$ denotes the $\e$-neighborhood of the identity in $N^-$ and $P^+$ respectively.
If $\e$ is sufficiently small, then for all $y\in W_\e^{s}(x)$ and $z\in W_\e^{cu}(x)$, $W_{2\e}^{cu}(y)$ and $W_{2\e}^{s}(z)$ meet in a single point.
Denoting this intersection by $\iota(y,z)$, the map 
\begin{equation}\label{eq.int}
\iota : W_\e^s(x)\times W_\e^{cu}(x)\to \Ga\ba G    
\end{equation}
is a homeomorphism in a neighborhood of $x$.

\begin{dfn}\label{dfn.bms}\normalfont
A locally finite $AM$-invariant Borel measure $\bms$ on $\Ga\ba G$ is called a \textit{Bowen-Margulis-Sullivan} $(\BMS)$ measure if there exists a linear form $\psi\in\frak a^*$ and a family of measures $\{\bms_x^s\}$ and $\{\bms_x^{cu}\}$ $(x\in\Ga\ba G)$ such that for all $x\in\Ga\ba G$,
\begin{enumerate}
    \item $\bms_x^s$ (resp. $\bms_x^{cu}$) is supported on $W^s(x)$ (resp. $W^{cu}(x)$);
    \item $\bms_x^s.a=e^{-\psi(\log a)}\bms_{x.a}^s$ for all $a\in A^+$;
    \item $\bms=\iota.(\bms_x^s\otimes \bms_x^{cu})$ when restricted to the image of $\iota$. 
\end{enumerate}
\end{dfn}
Note that we may view $\bms$ as a measure on $\Ga\ba G/M$ as well.
\subsection*{Finite $\BMS$ measures}
When $G$ is of rank one, the space $\Ga\ba G/M$ represents a unit tangent bundle of a rank one locally symmetric space, on which the one-dimensional group $A=\{a_t\}$ acts as a geodesic flow.
In this case, the construction of Sullivan \cite{Sul} 
which extends the idea of Patterson \cite{Pat}, 
shows that there exists a measure $\bms$
supported on the non-wandering set of $\{a_t\}$, and family of measures $\bms_x^s$, $\bms_x^{cu}$ verifying the conditions in Definition \ref{dfn.bms}. 

Moreover, one can produce many examples of infinite co-volume subgroups with finite $\BMS$ measures.
Examples include a large class of subgroups such as convex cocompact subgroups, and geometrically finite subgroups of rank one Lie groups \cite{Sul}.
And still, the list is far from being exhaustive, as the example of Dal'bo-Otal-Peign\'e \cite{DOP} suggests.

Turning to the case when $G$ is of higher rank, there always exists a $\BMS$ measure $\bms$ supported on $\Om_\Ga$.
In fact, Quint shows that one can find a $\BMS$ measure for every linear form $\psi\in\frak a^*$ which is tangent to the growth indicator $\psi_\Ga$ (see 
Proposition \ref{prop.Q}).
As a consequence, we have abundance of examples of a measure $\bms$ as in Definition \ref{dfn.bms}.
However, in contrast to the rank one case, 
there was no known example of discrete subgroups admitting a finite $\BMS$ measure except for lattices.

The following theorem is the main result of the paper, in which we classify all discrete subgroups admitting finite $\BMS$ measures on $\Ga\ba G$:


\begin{thm}\label{thm.BMS}
    If $\Ga\ba G$ admits a finite $\BMS$ measure, then there exist connected subgroups $G_0,G_1,\cdots,G_q$ with corresponding Zariski dense discrete subgroups $\Ga_0,\Ga_1,\cdots,\Ga_q$ such that the following holds:
\begin{enumerate}
    \item $G=G_0G_1\cdots G_q$ and $\Ga$ is commensurable with $\Ga_0\Ga_1\cdots\Ga_q$.
    \item $\Ga_0$ is a lattice in $G_0$.
    \item For all $1\leq i\leq q$, $G_i$ is a simple Lie group of rank one, and $\Ga_i\ba G_i$ admits a finite $\BMS$ measure.
\end{enumerate}
\end{thm}

The theorem leads to a new criterion detecting lattices among the discrete subgroups of higher rank Lie groups.
\begin{cor}\label{cor.lat}
    If $G$ is a connected semisimple real algebraic group with no rank one factor and $\Ga<G$ is a Zariski dense discrete subgroup, then  $\Ga$ is a lattice if and only if $\Ga\ba G$ admits a finite $\BMS$ measure, in which case the $\BMS$ measure is necessarily the Haar measure.
\end{cor}



 As another application,
 we obtain the following theorem, in a joint work with Edwards and Oh \cite{EFLO} using their previous work \cite{EO}:
 \begin{thm}[Edwards-Fraczyk-Lee-Oh \cite{EFLO}]
     Let $\Ga<G$ be as in Corollary \ref{cor.lat}.
     If $\Ga\ba G$ has infinite volume, then there is no positive $L^2$-integrable Laplace eigenfunction on the associated locally symmetric space $\Ga\ba X$.
 \end{thm}

Note that Theorem \ref{thm.BMS} also recovers Quint's theorem (Theorem \ref{thm.Q}), since there always exists a  $\BMS$ measure supported on $\Om_\Ga$.

\subsection*{Idea of the proof}
 We describe the proof in the case when $G$ is a simple higher rank group, or is of the form $G=G_1\times G_2$ and $\Ga$ is irreducible, i.e. has dense projections onto each factor. 
 
 To show that $\Ga$ is a lattice in $G$, one needs to verify that the volume of $\Ga\bs G$ is finite with respect to the Haar measure. We achieve this by proving that any finite $\BMS$ measure $\bms$ on $\Ga\bs G$ is actually Haar. The proof is based on the high entropy method of Einsiedler-Katok \cite{EK} and Einsiedler-Katok-Lindenstrauss \cite{EKL}, and the product structure of leafwise measures shown in \cite{EL}.
These methods allow to establish additional invariance of a finite $A$-invariant measure $\mu$ on a quotient $\Ga\ba G$ provided that sufficiently many elements in $A$ act on $(\Ga\ba G,\mu)$ with positive entropy.

If $G$ is a simple higher rank group, the presence of the local product structure of $\bms$ together with the product structure of leafwise measures $\bms_x^{N^\pm}$ of $\bms$ along $N^\pm$ (Proposition \ref{lem.pro}) shows that either $\bms_x^{N^\pm}$ is supported on a proper Zariski closed subset of $N^\pm$ or the leafwise measures along every root subgroups of $N^\pm$ are nontrivial.
The first case can be excluded solely using the Zariski density of $\Ga$.
In the second case we finish the proof by applying the high-entropy method.

Now assume that $G=G_1\times G_2$ and $\Ga$ is irreducible.
Associated with $\bms$, there are $\Ga$-conformal measures on $G/P^\pm$ underlying the definition of $\bms$ \eqref{eq.PS}.
The product structure of leafwise measure implies that the conformal measure on $G/P^{\pm}$ is itself a product of measures on $G_i/P_i^{\pm}$ where $P_i^\pm=G_i\cap P^\pm$ $(i=1,2)$. These measure are {\textquotedblleft}almost{\textquotedblright} conformal under the projections of $\Ga$ to $G_i$ $(i=1,2)$. Now, the density of the projections allows us to deduce that the measures are $G_i$-conformal, hence Lebesgue (Proposition \ref{lem.d}). It follows that $\bms$ is Haar.
\subsection*{Outline of the paper}
In Section \ref{sec.B} we gather preliminary results on semisimple Lie groups, Bowen-Margulis-Sullivan measures and the machinery of leafwise measures developed by Einsiedler, Katok and Lindenstrauss.  Section \ref{sec.C} is devoted to the study of leafwise measures of BMS-measures. Here, we identify the leafwise measures with respect to the unipotent radicals of minimal parabolic subgroups in terms of the Patterson-Sullivan measures of $\Ga$ (formula (\ref{eq.2})). We also show that when $\Gamma$ has a discrete projection onto a factor of $G$ the problem can be handled factor by factor (Lemma \ref{lem.disc}). In section \ref{sec.proof} we use the product structure of leafwise measures (Lemma \ref{lem.str}) and Zariski density of $\Ga$ to show that leafwise measures with respect to all root subgroups are non-trivial (Lemma \ref{lem.tr}). We then deduce the additional invariance of the BMS measure (Prop. \ref{lem.rk2} and Prop. \ref{lem.d}) and prove the main theorem. 
\subsection*{Acknowledgement}
We would like to thank Hee Oh for her helpful suggestions and a careful reading of the earlier version of the paper.
Her numerous comments have significantly improved the quality of the paper.
We also thank Manfred Einsiedler for his helpful comments. We thank the anonymous referees for their corrections and helpful remarks. 

\section{Preliminaries}\label{sec.B}
In this section, we will fix notations and recall backgrounds that will be used throughout the paper.

Let $G$ be a connected semisimple real algebraic group, and $\Ga<G$ be a Zariski dense discrete subgroup.
Let $\frak g$ denote the Lie algebra of $G$, $\Theta : \frak g\to\frak g$ be a Cartan involution, $\frak k$ and $\frak p$ respectively be $+1$, $-1$ eigenspace of $\Theta$ so that we have the Cartan decomposition $\frak g=\frak k\oplus\frak p$.
Fix a maximal abelian subspace $\frak a\subset\frak p$ and a closed positive Weyl chamber $\frak a^+\subset\frak a$. 
Let $A=\exp(\frak a)$, $A^+=\exp(\frak a^+)$, $\op{int}A^+$ be the interior of $A^+$, and $K$ be the maximal compact subgroup of $G$ whose Lie algebra is $\frak k$.
Let
\begin{align*}
N^\pm&=\{g\in G: a^{k} ga^{-k}\to e\text{ as }k\to\pm\infty\text{ for some }a\in \op{int}A^+\},
\end{align*}
be a pair of maximal horospherical subgroups.
We also let $P^\pm=MAN^\pm$ be the corresponding minimal parabolic subgroups, where $M$ is the centralizer of $A$ in $K$.

\subsection*{Conformal measures}
Let $\cal F$ be the set of all minimal parabolic subgroups of $G$, which can be identified with $G/P^\pm$.
We define $\sigma : G\times \cal F\to \frak a$ as follows: for $g\in G$ and $\xi\in\cal F$, $\sigma(g,\xi)\in\frak a$ is the unique element satisfying
\begin{equation}\label{eq.Iwa}
gk_\xi\in K\exp({\sigma(g,\xi)})N^-    
\end{equation}
where $k_\xi\in K$ is any element such that $\xi=k_\xi P^-$.
The definition does not depend on the choice of representative $k_\xi\in K$ for $\xi$, and $\sigma$ is called the \textit{Iwasawa cocycle}.

Given a linear form $\psi\in \mathfrak a^*$, a finite Borel measure $\nu$ on $\cal F$ is called a $(\Ga, \psi)$-\textit{conformal} 
measure if, for any $\ga\in \Ga$ and $\xi\in \F$, 
\begin{equation}\label{eq.PS}
 \frac{d(\ga. \nu)}{d\nu}(\xi) =e^{-\psi (\sigma (\ga^{-1}, \xi))},   
\end{equation}
 where $\ga. \nu(E)=\nu(\gamma^{-1} E)$ for any Borel subset $E\subset \F$.
We will refer to $\nu$ as a $\Ga$-conformal measure if it is a $(\Ga,\psi)$-conformal measure for some $\psi\in\frak a^*$.

\subsection*{Bowen-Margulis-Sullivan measures} 
Let $\op{N}_K(\frak a)$ and $\op{Z}_K(\frak a)$ respectively denote the normalizer and the centralizer of $\frak a$ in $K$.
The Weyl group $\cal W$ is then defined as
\begin{equation}\label{eq.Weyl}
\cal W=\op{N}_K(\frak a)/\op{Z}_K(\frak a).    
\end{equation}
Let $\i:\frak a\to\frak a$ be the opposition involution; it is given by $\i=-\op{Ad}_{w_0}$ where $w_0$ is the longest element of the Weyl group $\cal W$.

Let $\psi\in\frak a^*$ be a linear form and $\nu_\psi$, $\nu_{\psi\circ\i}$ be a pair of 
$\Ga$-conformal measures on $\cal F$ associated to $\psi, \psi\circ\i$.
By means of the Hopf parametrization
\begin{align}\label{eq.Hopf}
G/M&\to G/P^-\times G/P^+\times\frak a\\
gM&\mapsto (gP^-, gP^+, \sigma(g,P^-)),\notag
\end{align}
the \textit{Bowen-Margulis-Sullivan} $(\BMS)$ measure 
$\tilde \bms _{\nu_\psi,\nu_{\psi\circ\i}}$ 
is defined to be the locally finite Borel measure on $G/M$ given by:
\begin{equation}\label{eq.BMS0}
  \tilde \bms_{\nu_\psi, \nu_{\psi\circ\i}}=e^{\psi( \cal G(\cdot,\cdot))}\nu_\psi\otimes \nu_{\psi\circ\i}\otimes  \op{Leb}_{\frak a}
\end{equation}
 where $\op{Leb}_{\frak a}$ denotes the Lebesgue density on $\mathfrak a$, and $\cal G$ is the \textit{vector valued Gromov product} defined by
 \begin{equation*}
\cal G(gP^-,gP^+)=\sigma(g,P^-)+\i \,\sigma(g,P^+).     
 \end{equation*}
By the relation \eqref{eq.PS}, $\tilde \bms_{\nu_\psi, \nu_{\psi\circ\i}}$ is left $\Ga$-invariant and descends to a measure on $\Ga\ba G/M$, which we denote by $\bms_{\nu_\psi, \nu_{\psi\circ\i}}$.

The fact that this definition of $\BMS$ measure is equivalent to that given in the introduction, can be derived from the properties shown in \cite[Sec. 4.1]{ELO}.

\subsection*{Limit cone and the growth indicator}
By Cartan decomposition $G=KA^+K$, for every $g\in G$, there exists a unique element $\mu(g)\in\frak a^+$ such that
$$
g\in K\exp(\mu(g)) K.
$$

Benoist introduced the \textit{limit cone} of $\Ga$ \cite{Ben}, denoted by $\cal L_\Ga$,
which is the asymptotic cone of $\{\mu(\ga):\ga\in\Ga\}$.
He showed that if $\Ga$ is Zariski dense, then $\cal L_\Ga\subset \frak a^+$ is a closed convex cone with nonempty interior.

The \textit{growth indicator} of $\Ga$, denoted by $\psi_\Ga$, is a function $\psi_\Ga : \frak a^+\to \bb R_{\geq 0}\cup\{-\infty\}$ defined by Quint \cite{Quint3} as
\begin{equation*}\label{eq.GI}
\psi_\Ga(v)=\norm{v} \inf_{v\in\cal C}h_{\cal C};    
\end{equation*}
here $\norm{\cdot}$ denotes the norm on $\frak a$ induced from the Killing form on $\frak g$, infimum is taken over all open cone $\cal C\subset \frak a^+$ containing $v$, and $h_{\cal C}$ denotes the abscissa of convergence for the series $s\mapsto\sum_{\ga\in\Ga, \mu(\ga)\in\cal C} e^{-s\norm{\mu(\ga)}}$.
We say a linear form $\psi\in\frak a^*$ is \textit{tangent to} $\psi_\Ga$ at $v\in\frak a^+$ if
\begin{equation*}\label{eq.tang}
\text{$\psi\geq \psi_\Ga$ on $\frak a^+$ and $\psi(v)=\psi_\Ga(v)$
}    
\end{equation*}
The following general properties of $\psi_\Ga$ are known (\cite{Quint1, Quint3}):
\begin{prop}[Quint]\label{prop.Q}
    We have:
    \begin{enumerate}
        \item For any $v\in\frak a^+$, $\psi_\Ga(v)\geq 0$ if and only if $v\in\cal L_\Ga$.
        Moreover, $\psi_\Ga$ is strictly positive in the interior of $\cal L_\Ga$.
        \item If there exists a $(\Ga,\psi)$-conformal measure on $\cal F$, then $\psi\geq \psi_\Ga$.
        \item For any $\psi\in\frak a^*$ tangent to $\psi_\Ga$ at $v\in\op{int}\frak a^+$, there exists a $(\Ga,\psi)$-conformal measure on $\cal F$.
    \end{enumerate}
\end{prop}

\subsection*{Conditional measures}
Let $\cal B$ be the Borel $\sigma$-algebra on $\Ga\ba G$ and $\cal A\subset\cal B$ be a countably generated sub $\sigma$-algebra.
For $x\in\Ga\ba G$, let $[x]_{\cal A}$ denote the $\cal A$-atom of $x$, which is the intersection of all elements of $\cal A$ containing $x$.

Given a finite Borel measure $\mu$ on $\Ga\ba G$, there exists a $\mu$-conull set $X\subset\Ga\ba G$ and a family of finite measures $\{\mu_x^{\cal A}:x\in X\}$ of total mass one with the following properties  \cite[Thm. 5.9]{EL}:
\begin{enumerate}
    \item $\mu_x^{\cal A}$ is supported on $[x]_{\cal A}$ and $\mu_x^{\cal A}=\mu_{y}^{\cal A}$ if $[x]_{\cal A}=[y]_{\cal A}$,
    \item the map $x\mapsto \mu_x^{\cal A}(B)$ is $\cal A$-measurable for all $B\in\cal B$,
    \item and we have
    \begin{equation}\label{eq.ED}
\mu=\int_{\Ga\ba G} \mu_x^{\cal A}\,d\mu(x).
\end{equation}    
\end{enumerate}
They are unique in the following sense: if $\la_x^{\cal A}$ is another family of measures satisfying the   above, then there exists a $\mu$-conull set $X'\subset\Ga\ba G$ such that $\mu_x^{\cal A}=\la_x^{\cal A}$ for all $x\in X'$.

\subsection*{Leafwise measures}
For a closed subgroup $S<G$, there exists a $\mu$-conull set $X'\subset \Ga\ba G$ and a family $\{\mu_x^S:x\in X'\}$ of locally finite Borel measures on $S$, well-defined up to a proportionality called the \textit{leafwise measure} of $\mu$ along $S$ \cite[Thm. 6.3]{EL}.
Locally, they can be described as follows:

For $R>0$, let $B_S(R)$ denote the ball of radius $R$ in $S$ centered at $e$.
For any $x\in\Ga\ba G$ such that the map $s\mapsto x.s$ $(s\in S)$ is injective, there exists a local cross section $C\subset G$ containing $e$ such that the map 
 \begin{align*}
 \phi:C\times B_S(R)&\to\Ga\ba G\\
 (g,s)&\mapsto x.(gs)
 \end{align*}
 is injective, where $C$ (and hence $\phi$) depends on $R$ and $x$.
Let $\cal{B}_C$ denote the Borel $\sigma$-algebra on $\phi(C\times\{e\})$ induced from $\cal B$, and $\{\emptyset,S\}$ the trivial $\sigma$-algebra on $S$.
Up to proportionality, $\mu_x^S$ is then characterized\footnote{This follows from the defining equation \cite[(6.20b)]{EL} of $\mu_x^S$.
For the place where local cross sections are considered, see \cite[Def. 6.6]{EL}.
}
as follows:
\begin{equation}\label{eq.leaf}
\phi(e,\cdot)_*\mu_x^{S}|_{B_S(R)}\propto (\mu|_{\Im\phi})_x^{\cal B_{\cal C}\otimes\{\emptyset,S\}}.   
\end{equation}

Now let $\mu$ be a finite $A$-invariant measure.
We record here several results concerning the leafwise measures of $\mu$ that will be needed.
The main reference is \cite{EL}.
For a closed subgroup $S<G$, we have:
\begin{lem}\label{lem.ad}\cite[Lem. 7.16]{EL}
If $A$ normalizes $S$, then for all $a\in A$,
$$
\op{\theta}_* \mu_{x}^{S}\propto\mu_{x.a}^{S}
$$
for $\mu$-a.e. $x$ where $\theta : S\to S$ is the map given by $\theta(s)=a^{-1}sa$.
\end{lem}
A leafwise measure will be called \textit{trivial} if it is proportional to a Dirac measure supported at $e$.
By Lemma \ref{lem.ad}, if $\mu$ is a finite $A$-invariant, ergodic measure then either $\mu_x^{S}$ is trivial $\mu$-a.e. or nontrivial $\mu$-a.e.
\begin{lem}\cite[Lem. 9.18]{EL}\label{lem.inv}
    Let $L\leq S$ be a closed subgroup.
    Suppose that for every $\ell\in L$, $\ell.\mu_x^{S}=\mu_x^{S}$ for $\mu$-a.e. $x$,
    Then $\mu$ is $L$-invariant.
\end{lem}
In particular, $\mu$ is $S$-invariant if and only if $\mu_x^{S}$ is Haar $\mu$-a.e.
For a fixed unit ball $B_1$ in $N^+$ and $a\in A^+$, define
$$
D_{\mu}(a,N^+)(x):=\lim_{k\to\infty}\frac{1}{k}\log\mu_x^{N^+}(a^{-k}B_1a^k).
$$
\begin{lem}\cite[Thm. 7.6]{EL}\label{lem.ent}
For any $a\in A^+$,
$D_{\mu}(a,N^+)(x)$ is defined $\mu$-a.e., and the measure theoretic entropy of $a$ with respect to $\mu$ is given by
    \begin{equation}\label{eq.EL}
h_\mu(a)=\int_{\Ga\ba G} D_\mu(a,N^+)(x)\,d\mu(x).
\end{equation}
\end{lem}

For $a\in A$, let $\op{Z}_G(a)$ denote the centralizer of $\langle a\rangle $.
The proposition below describes the product structure of the leafwise measures.
Together with the high-entropy method of \cite{EK} and \cite{EKL}, this is the most important ingredient of our proof.
\begin{prop}\cite[Cor. 8.8]{EL}\label{lem.pro}
    Let $U<G$ be a subgroup normalized and contracted by an element $ a\in A$.
    Assume that $T$ is a subgroup of $\op{Z}_G(a)$ normalizing $U$ and $H=T\ltimes U$.
    Then for any finite $\langle a\rangle $-invariant measure $\mu$ on $\Ga\ba G$,
    $$
    \mu_x^{H}\propto \iota(\mu_x^{T}\times\mu_x^{U})
    $$
    for $\mu$-a.e. $x$, where $\iota : T\times U\to H$ is the product map $(t,u)\mapsto tu$.
\end{prop}

\section{Ergodic properties of $\BMS$ measures}\label{sec.C}
Let $G$ be a connected semisimple real algebraic group and $\Ga<G$ be a Zariski dense discrete subgroup.
We fix $\psi\in\frak a^*$ and a pair of $\Ga$-conformal measures $\nu_\psi$, $\nu_{\psi\circ\i}$ on $\cal F$ and let
\begin{equation}\label{eq.sfm}
\mathsf{m}=\bms_{\nu_\psi, \nu_{\psi\circ\i}}    
\end{equation}
denote the associated $\BMS$ measure on $\Ga\ba G$.

\subsection*{Conservative $\BMS$ measures}
Let $\{a_t\}\subset A$ be a one-parameter subgroup.
We recall the following definitions:
\begin{enumerate}
    \item A Borel subset $B\subset\Ga\ba G$ is a \textit{wandering set}, if for $\bms$-a.e. $x\in B$,
$$
\int_{\bb R}{\mathbbm{1}_B(xa_t)}\,dt<\infty.
$$
    \item We say $(\Ga\ba G,\bms,\{a_t\})$ is \textit{conservative} if there exists no wandering set $B$ such that $\bms(B)>0$.
\end{enumerate}
Note that if $\bms(\Ga\ba G)<\infty$, then $(\Ga\ba G,\bms,\{a_t\})$ is conservative by Poincar\'e recurrence theorem.

For any $\BMS$ measure on $\Ga\ba G$, the following dichotomy is known:
\begin{lem}\cite{BLLO}\label{lem.erg}
    Let $\{a_t\}\subset A$ be a one-parameter subgroup containing an element of $\op{int}A^+$.
    Then $(\Ga\ba G,\bms,\{a_t\})$ is conservative if and only if $(\Ga\ba G,\bms,\{a_t\}\times M)$ is ergodic.
    In particular, if $\bms(\Ga\ba G)<\infty$, then it is $AM$-ergodic.
\end{lem}

\subsection*{Limit set}
Let $C(\cal F)$ be the set of all continuous functions on $\cal F$.
The \textit{limit set} $\La_\Ga\subset \cal F$ is defined by
\begin{equation*}\label{eq.lim}
\La_\Ga=\{\xi\in\cal F: \ga_i.\nu\to \delta_\xi\text{ for some }\ga_i\to\infty\text{ in }\Ga\}    
\end{equation*}
where $\nu$ denotes the unique finite $K$-invariant measure on $\cal F$ of total mass one, $\delta_\xi$ denotes the Dirac mass at $\xi$, and the convergence takes place in the weak-star topology on the dual of $C(\cal F)$.
Since $\Ga$ is Zariski dense, $\La_\Ga$ is the unique $\Ga$-minimal subset of $\cal F$, and is Zariski dense in $\cal F$ \cite[Lem. 3.6]{Ben}.
If $\{a_t\}$ is a one-parameter subsemigroup containing an element of $\op{int}A^+$, then for $g\in G$, the following can be checked from the definition above:
\begin{equation}\label{eq.gat}
    \text{if $[g]a_{t_i}\in\Ga\ba G$ is bounded for some $t_i\to+\infty$, then $gP^{-}\in\La_\Ga$.}
\end{equation}
\subsection*{The set $\Om_\Ga$}
Let $\cal F^{(2)}$ denote the $G$-orbit of $(P^-,P^+)\in \cal F\times\cal F$; it is the unique open $G$-orbit in $\cal F\times\cal F$.
We define
\begin{equation}\label{eq.La2}
\La_{\Ga}^{(2)}=(\La_{\Ga}\times\La_{\Ga})\cap\cal F^{(2)},\quad\text{ and }\quad\tilde\Om_\Ga=\La_\Ga^{(2)}\times\frak a.    
\end{equation}
The image of the embedding \eqref{eq.Hopf} is given by $\cal F^{(2)}\times\frak a$, and hence we may consider $\tilde\Om_\Ga$ as a subset of $G/M$.
The set $\Om_\Ga\subset \Ga\ba G/M$ in the introduction (see Theorem \ref{thm.Q}) can then be identified with $\Ga\ba\tilde\Om_\Ga$.

In the proof below, for $w\in\cal W$ \eqref{eq.Weyl}, we will use the expressions 
\begin{equation}\label{eq.waw}
waw^{-1} (a\in A),\quad wP^\pm,    
\end{equation}
where $w$ should be understood as a group element in $\op{N}_K(\frak a)$ representing it; note that they do not depend on the choice of the representatives.
We have:
\begin{lem}\label{lem.Om}
    If $\bms(\Ga\ba G)<\infty$, then $\bms$ is supported on $\Om_\Ga$.
\end{lem}
\begin{proof}
Let $\{a_t\}$ be a fixed one-parameter subsemigroup containing an element of $\op{int}A^+$ and $\{w a_t w^{-1}\}$ be subsemigroups for each $w\in\cal W$.
Since $\bms$ is finite, it is conservative with respect to $\{w a_t w^{-1}\}$, and hence
$$
X_w:=\{[g]\in\Ga\ba G: [g](wa_{t_i}w^{-1})\text{ is bounded for some }t_i\to+\infty\}
$$
is $\bms$-conull.
Setting $X_0=\cap\{ X_w:w\in\cal W\}$, \eqref{eq.gat} implies that $gwP^-\in\La_\Ga$ for all $[g]\in X_0$ and $w\in\cal W$.
In particular, $gP^\pm\in\La_{\Ga}$ for all $[g]\in X_0$ and the lemma follows in view of the Hopf parametrization \eqref{eq.Hopf} and \eqref{eq.La2}.
\end{proof}

\subsection*{Leafwise measures of $\bms$ along $N^+$}
The leafwise measures $\bms_{x}^{N^+}$ of the $\BMS$ measure $\bms$ along $N^+$, can be explicitly described from \eqref{eq.BMS0};
their densities are given by
\begin{equation}\label{eq.2}
d\bms_{x}^{N^+}(n)=e^{\psi(\sigma(gn,P^-))}d\nu_{\psi}(gnP^-),
\end{equation}
where $g\in G$ is such that $x=[g]$.
The expression \eqref{eq.2} does not depend on the choice of the representative $g$, by the relation \eqref{eq.PS}.

Let $a\in A^+$.
Now a direct computation as in \cite[Lem. 4.2]{ELO} shows
$$
\bms_x^{N^+}(a^{-k}B_1a^k)=e^{k\cdot\psi(\log a)}\bms_{xa^k}^{N^+}(B_1).
$$
Since $\{xa^k:k\in\bb N\}$ is recurrent for $\bms$-a.e. $x$ by Poincar\'e recurrence,
$$
D_{\bms}(a,N^+)(x)=\lim_{k\to\infty}\frac{1}{k}\log \bms_x^{N^+}(a^{-k}B_1a^k)=\psi(\log a),
$$
for $\bms$-a.e. $x$.
By Lemma \ref{lem.ent},
we have
\begin{equation}\label{eq.ma}
h_\bms(a)
=\psi(\log a).
\end{equation}
\begin{rmk}\label{lem.pos}
    If $\bms(\Ga\ba G)<\infty$, then for all $a\in\op{int}A^+$,
    $$
    h_\bms(a)>0.
    $$
\end{rmk}
\begin{proof}
    If $\bms(\Ga\ba G)<\infty$, then $(\Ga\ba G, \bms)$ is conservative for the action of any one-parameter subgroup of $A$.
    This together with \cite[Lem. 4.6]{Ben} implies that the limit cone $\cal L_\Ga$ coincides with $\frak a^+$.
    On the other hand, since $\psi\in\frak a^*$ is a linear form for which there exists a $(\Ga,\psi)$-conformal measure on $\cal F$, we have $\psi\geq \psi_\Ga$ by 
    Proposition \ref{prop.Q}(2).
    Since $\psi_\Ga$ is strictly positive in the interior of $\cal L_\Ga$ 
    by Proposition \ref{prop.Q}(1),
    so is $\psi$, and hence the lemma follows by \eqref{eq.ma}.
\end{proof}
\begin{rmk}\label{rmk.PE}
\normalfont
    Remark \ref{lem.pos} provides many elements of $a\in A^+$ such that $h_{\bms}(a)>0$.
    By a Pesin \cite{Pes} type entropy formula obtained in \cite{EL}, this information can be used to prove the existence of a root subgroup $U< N^+$ for which $\bms_x^{U}$ is nontrivial $\bms$-a.e.
    However, the proof of Theorem \ref{thm.BMS} does not rely on the calculation of entropy.
    Our proof requires the knowledge of nontriviality of $\bms_x^U$ for \textit{all} root subgroups $U<N^+$ (Lemma \ref{lem.tr}) and entropy formula as in \cite{Pes} will not be sufficient to draw this conclusion.
\end{rmk}

\subsection*{$\BMS$ measures on the product $G_1\times G_2$}
Consider the case
$$
G=G_1\times G_2
$$
where $G_1$ and $G_2$ are semisimple real algebraic groups without compact factors.
We will also make the identifications $G_1\simeq G_1\times\{e\}$ and $G_2\simeq \{e\}\times G_2$.
Let $\pi_i : G\to G_i$ denote the projection map, and set $\Ga_i:=\pi_i(\Ga)$ $(i=1,2)$, to ease the notation.
Denote by $\La_{\Ga_i}$ the limit set of $\Ga_i$ in $\cal F_i:=G_i/P_i^-$, where $P_i^\pm=G_i\cap P^\pm$ $(i=1,2)$.
Note that $\Ga_i$ is Zariski dense in $G_i$ and hence $\La_{\Ga_i}$ is the unique $\Ga_i$-minimal set in $\cal F_i$ by \cite[Lem. 3.6]{Ben}.

Let $p_1 : G/P^-\to G_1/P_1^-$ be the projection map.
For a $\Ga$-conformal measure $\nu$ supported on $\La_{\Ga}$, let
\begin{equation}\label{eq.nu}
\nu=\int_{\La_{\Ga_1}}\nu_{\xi_1}\,d(p_{1*}\nu)(\xi_1)    
\end{equation}
where $\nu_{\xi_1}$ denotes the conditional measure of $\nu$ along the fiber $p_1^{-1}(\xi_1)$;
this is obtained from \eqref{eq.ED} by taking $\mu=\nu$, and 
$\cal A=p_1^{-1}(\cal B_1)$ where $\cal B_1$ denotes the Borel $\sigma$-algebra on $G_1/P_1^-$.
A direct computation shows:
\begin{lem}\label{lem.E}
We have the following:
\begin{enumerate}
    \item $p_{1*}\nu$ is a $\Ga\cap G_1$-conformal measure whose support is contained in $\La_{\Ga_1}$;
    \item for $p_{1*}\nu$-a.e. $\xi_1\in\La_{\Ga_1}$, $\nu_{\xi_1}$ is a $\Ga\cap G_2$-conformal measure.
\end{enumerate}
\end{lem}
\begin{proof}
    Let $\psi\in\frak a^*$ be the linear form associated to $\nu$ and $\frak a_i\subset\frak a$ be the Lie algebra of $A\cap G_i$ $(i=1,2)$.
    We can write $\sigma$ uniquely as a sum $\sigma_1+\sigma_2$ of $\frak a_i$-valued cocyles $\sigma_i:G_i\times G_i/P_i^-\to\frak a_i$.
    Note that $\sigma(g_i,\cdot)=\sigma_i(g_i,p_i(\cdot))$ for all $g_i\in G_i$.
    Let $\mathbbm{1}_{G_i/P_i^-}$ denote the constant function on $G_i/P_i^-$ which is identically 1.
    
    For (1), let $\ga\in\Ga\cap G_1$ be arbitrary and $f\in C(G_1/P_1^-)$ be a continuous function.
    By \eqref{eq.PS},
    \begin{align*}
    &\ga.(p_{1*}\nu)(f)= (p_{1*}\nu)(f(\ga\,\cdot))=\nu((f\otimes\mathbbm{1}_{G_2/P_2^-})(\ga\,\cdot))\\
    &=\int_{G/P^-}e^{-\psi(\sigma(\ga^{-1},\xi))}f(p_1(\xi))\,d\nu(\xi)\\
    &=\int_{G/P^-}e^{-\psi|_{\frak a_1}(\sigma_1(\ga^{-1},p_1(\xi)))}f(p_1(\xi))\,d\nu(\xi)\\
    &=\int_{G_1/P_1^-}e^{-\psi|_{\frak a_1}(\sigma_1(\ga^{-1},\xi_1))}f(\xi_1)\,d(p_{1*}\nu)(\xi_1)
    \end{align*}
    This implies that $p_{1*}\nu$ is a $(\Ga\cap G_1, \psi|_{\frak a_1})$-conformal measure.
    Recall that $\La_\Ga$ (resp. $\La_{\Ga_1}$) is the unique closed $\Ga$ (resp. $\Ga_1$)-invariant subset of $G/P^-$ (resp. $G_1/P_1^-$).
    Since $p_1$ is equivariant, it follows that $p_1(\La_{\Ga})=\La_{\Ga_1}$ and $p_{1*}\nu$ is supported on $\La_{\Ga_1}$.

    For (2), let $\ga\in\Ga\cap G_2$ be arbitrary and $f\in C(G_2/P_2^-)$ be a continuous function.
    Similarly as in (1), we have
    \begin{align*}
    &\ga.\nu(\mathbbm{1}_{G_1/P_1^-}\otimes f)=\int_{G/P^-}e^{-\psi|_{\frak a_2}(\sigma_2(\ga^{-1},p_2(\xi)))}f(p_2(\xi))\,d\nu(\xi)\\
    &=\int_{\La_{\Ga_1}}\nu_{\xi_1}(e^{-\psi|_{\frak a_2}(\sigma_2(\ga^{-1},\cdot))}f(\cdot))\,d(p_{1*}\nu)(\xi_1),
    \end{align*}
    where the first equality is due to \eqref{eq.PS} and the second equality is due to \eqref{eq.nu}.
    On the other hand, applying \eqref{eq.nu} in a different way gives
    $$
    \ga.\nu(\mathbbm{1}_{G_1/P_1^-}\otimes f)=\nu(\mathbbm{1}_{G_1/P_1^-}\otimes f(\ga\,\cdot))=\int_{\La_{\Ga_1}}\ga.\nu_{\xi_1}(f)\,d(p_{1*}\nu)(\xi_1).
    $$
    Since $f$ was arbitrary, (2) follows from the uniqueness of the conditional measures \eqref{eq.ED} by comparing the above identities.
\end{proof}
Now suppose that $\bms(\Ga\ba G)<\infty$, and one of $\Ga_i$, say $\Ga_1$ is discrete.
Consider the projection $\Ga\ba \tilde\Om_\Ga\to\Ga_1\ba \tilde\Om_{\Ga_1}$ induced from the map $\Ga\ba G\to \Ga_1\ba G_1$, which we will call $\pi_1$ by abuse of notation.
Similarly as in \eqref{eq.nu}, let
$$
\bms=\int_{\Ga\ba\tilde\Om_{\Ga_1}}\bms_{x_1}\,d(\pi_{1*}\bms)(x_1)
$$
where $\bms_{x_1}$ denotes the conditional measure of $\bms$ along the fiber $\pi_1^{-1}(x_1)$.
    More precisely, denoting by $\cal B$ the Borel sigma-algebra of $\Ga_1\ba G_1$ and $\cal A$ the smallest countably generated sigma-algebra equivalent to $\pi_1^{-1}(\cal B)$ (cf. \cite[Def. 5.7, Prop. 5.8]{EL}), $\bms_{x_1}$ is obtained by taking $\mu=\bms$ from \eqref{eq.ED}.
We have:
\begin{lem}\label{lem.fib}
    Assume that $\bms(\Ga\ba G)<\infty$ and $\Ga_1$ is discrete.
    Then $\Ga\cap G_2$ is Zariski dense in $G_2$, and
    $\bms_{x_1}$ is
    isomorphic to a finite $\BMS$ measure on $\Ga\cap G_2\ba \tilde\Om_{\Ga_2}$ for $\pi_{1*}\bms$-a.e. $x_1\in\Ga_1\ba \tilde\Om_{\Ga_1}$.
\end{lem}
\begin{proof}
    Since $\Ga_1$ is discrete, $\Ga\cdot G_2=\Ga_1\times G_2$ is a closed subset of $G=G_1\times G_2$.
    Hence the map $\Ga\cap G_2\ba G_2\to \Ga\ba G$ is a proper, closed embedding. 
    Recall that for $\pi_{1*}\bms$-a.e. $x_1$, $\bms_{x_1}$ is supported on an atom of $\cal A$.
    Because fibers of $\pi_1$ are closed $G_2$-orbits, we have $\cal A=\{\pi_1^{-1}(B):B\in\cal B\}$ and atoms of $\cal A$ are $G_2$-orbits.
    Hence each $\bms_{x_1}$ can be identified with a measure on $\Ga\cap G_2\ba G_2$.
    
    Next, observe that $\Ga\cap G_2$ is a normal subgroup of $\Ga_2$.
    Since $\Ga_2$ is Zariski dense in $G_2$, the Zariski closure of $\Ga\cap G_2$ is also a normal subgroup of $G_2$.
    Hence we may write $G_2$ as an almost direct product $G_2=G'\cdot G''$ where both $G'$ and $G''$ are products of $\bb R$-simple factors of $G_2$, $\Ga\cap G'$ is a Zariski dense subgroup of $G'$, and $\Ga\cap G''$ is a finite group.
    
    We claim that $G''=\{e\}$ and hence $G'=G_2$.
    Assume to the contrary that $G''$ is nontrivial, in particular $A\cap G''$ is an unbounded subgroup of $A$.
    Since $\bms$ is a finite $A$-invariant measure on $\Ga\ba G$, by Poincar\'e recurrence, we can find $\ga_i\in\Ga$, $g\in G$, and $a_i\to\infty$ in $A\cap G''$ such that $\ga_i ga_i\in G$ is bounded.
    Let us write $\ga_i=\ga_{i,1}\ga_{i,2}$, where $\ga_{i,1}\in G_1\cdot G'$ and $\ga_{i,2}\in G''$.
    Because $a_i\in G''$, passing to a subsequence, $\ga_{i,1}\in G_1\cdot G'$ is constant.
    Hence, we can assume that $\ga_{i,1}=\ga_{0,1}$ for all $i$.
    Now note that $\ga_0^{-1}\ga_i\in \Ga\cap G''$ and $(\ga_0^{-1}\ga_i)ga_i\in G$ is bounded.
    This is a contradiction, since $\ga_0^{-1}\ga_i$ provides an unbounded (hence infinite) sequence of elements in $\Ga\cap G''$.
    
    By the claim, $\Ga\cap G_2$ is a normal subgroup of $\Ga_2$ which is Zariski dense in $G_2$.
    It follows that $\La_{\Ga\cap G_2}=\La_{\Ga_2}$.
    Therefore, we have $\tilde\Om_{\Ga\cap G_2}=\tilde\Om_{\Ga_2}$.
    Hence that $\bms_{x_1}$ is a $\BMS$ measure in this case, follows from Lemma \ref{lem.E}(2).
\end{proof}

In the proof of the lemma below, we will use the notation \eqref{eq.waw}.
\begin{lem}\label{lem.GG}
   If $\bms(\Ga\ba G)<\infty$, then $\La_{\Ga}=\La_{\Ga_1}\times\La_{\Ga_2}$.
\end{lem}
\begin{proof}
    It suffices to show $\La_\Ga=\La_1\times\La_2$ for some closed non-empty $\Ga_i$-invariant subset $\La_i\subset \cal F_i$ $(i=1,2)$, since then
    $$
\La_{\Ga}\subset \La_{\Ga_1}\times\La_{\Ga_2}\subset \La_1\times\La_2=\La_\Ga,
    $$
    where the second inclusion is due to the fact that $\La_{\Ga_i}$ is the unique $\Ga_i$-minimal subset of $\cal F_i$ \cite[Lem. 3.6]{Ben} and $\La_i$ would clearly be a non-empty closed $\Ga_i$-invariant set.
    
    As in the proof of Lemma \ref{lem.Om}, for each $w\in\cal W$, set
$$
X_w:=\left\{[g]\in\Ga\ba G: 
\begin{array}{c}
[g](wa_{t}w^{-1})\text{ is recurrent for some one-parameter}\\
\text{subsemigroup }\{a_t\}\text{ containing an element of $\op{int}A^+$}
\end{array}
\right\}
$$
    and $X_0=\cap\{X_w:w\in\cal W\}$.
    Then $X_w$ is an $\bms$-conull set by the conservativity of $\bms$, and
    \begin{equation}\label{eq.w}
    \text{$gwP^-\in\La_\Ga$ for all $[g]\in X_0$ and $w\in\cal W$.}    
    \end{equation}
    Let $\xi=(\xi_1,\xi_2)$ and $\eta=(\eta_1,\eta_2)$ be the coordinates in $\cal F_1\times\cal F_2$ of arbitrary elements $\xi,\eta\in\La_\Ga$.
    We claim that 
    $(\xi_1,\eta_2)$ and $(\eta_1,\xi_2)$ belong to $\La_\Ga$; note that this implies $\La_\Ga$ is a product set.
    Recalling the definition \eqref{eq.La2}, since $\La_\Ga^{(2)}$ is dense in $\La_\Ga\times\La_\Ga$, it suffices to check the claim for $(\xi,\eta)\in \La_\Ga^{(2)}$.
    Now \cite[Lem. 3.6(iv)]{Ben} shows that there exists $g_j\in G$ such that $[g_j]\in X_e\cap X_{w_0}$ and $(g_jP^-,g_jP^+)\to (\xi,\eta)$ as $j\to\infty$, where $w_0\in\cal W$ denotes the longest element.
    Since $X_0$ is conull, its closure contains $\Om_\Ga$ and therefore contains $X_e\cap X_{w_0}$ (Lemma \ref{lem.Om}).
    Hence after modifying $g_j$, we may assume in addition that $[g_j]\in X_0$.
    Let us write $g_j=(g_{j,1},g_{j,2})\in G_1\times G_2$.
    To summarize, we have:
    $$
    (g_{j,i}P_i^-,g_{j,i}P_i^+)\to (\xi_i,\eta_i)\text{ as $j\to\infty$ for $i=1,2$.}
    $$
    Now the claim follows by considering $w\in\cal W$ such that 
    $$
    \text{$g_{j}wP^+=(g_{j,1}P_1^-,g_{j,2}P_2^+)$, and $g_{j}wP^-=(g_{j,1}P_1^+,g_{j,2}P_2^-)$,}
    $$
    in view of the property \eqref{eq.w}.    
\end{proof}

\begin{lem}\label{lem.disc}
     Assume that $\bms(\Ga\ba G)<\infty$ and $\Ga_1$ is discrete.
     Then $\Ga_2$ is also discrete, $\Ga$ is commensurable with $\Ga_1\times\Ga_2$ and $\bms$ is a finite extension of a product of $\BMS$ measures on $\Ga_1\ba G_1$ and $\Ga_2\ba G_2$.
\end{lem}
\begin{proof}
    As $\bms$ is finite and $\Ga_1$ is discrete, by Lemma \ref{lem.fib}, $\Ga\cap G_2$ is Zariski dense in $G_2$ and $\Ga\cap G_2\ba \tilde\Om_{\Ga_2}$ admits a finite $\BMS$ measure. 
    We claim that $\Ga_2$ is discrete.
    Observe that $\Ga_2$ normalizes $\Ga\cap G_2$, and so does its closure $\cl{\Ga_2}$, which is a closed Lie subgroup by Cartan's theorem.
    Let $N_0\subset\op{N}_{G_2}(\Ga\cap G_2)$ be the connected component of $\cl{\Ga_2}$.
    Since $\Ga\cap G_2$ is discrete, $N_0$ centralizes $\Ga\cap G_2$. 
    Hence $N_0=\{e\}$ and the claim follows.
    We now apply Lemma \ref{lem.fib} with $\Ga_2$ in place of $\Ga_1$.
    This in turn implies $\Ga\cap G_1\ba \tilde\Om_{\Ga_1}$ admits a finite $\BMS$ measure.
    Since $\tilde \Om_\Ga= \tilde\Om_{\Ga_1}\times \tilde\Om_{\Ga_2}$ by Lemma \ref{lem.GG}, we have 
    $$
    \Ga\ba \tilde\Om_\Ga=(\Ga\ba \tilde\Om_{\Ga_1}\times \tilde\Om_{\Ga_2})\simeq (\Ga_1\ba \tilde\Om_{\Ga_1}\times \Ga_2\ba \tilde\Om_{\Ga_2})\times (\Ga\ba \Ga_1\times\Ga_2),
    $$
    a measurable isomorphism which accounts for the disintegration of $\bms$ over the covering map 
    $$
    \Ga\ba \tilde\Om_\Ga \to (\Ga_1\ba \tilde\Om_{\Ga_1}\times \Ga_2\ba \tilde\Om_{\Ga_2}).
    $$
    Finiteness assumption on $\bms$ implies that $\#(\Ga\ba \Ga_1\times\Ga_2)<\infty$, and hence the lemma.
\end{proof}

We will need the following well-known fact:
\begin{lem}\label{lem.DD}
Let $\Delta<G$ be a Zariski dense subgroup.
If $G$ is simple, then $\Delta$ is either discrete or dense in the analytic topology.
\end{lem}
\begin{proof}
Let $H:=\cl{\Delta}$ be the closure of $\Delta$ in the analytic topology; it is a closed Lie subgroup of $G$ by Cartan's theorem.
Let $\frak h$ be the Lie algebra of $H$.
Its normalizer $\op{N}_G(\frak h)$ is a Zariski closed subgroup containing $\Delta$.
Since $\Delta$ is Zariski dense, it follows that $\op{N}_G(\frak h)$ is $G$.
Hence either $\frak h=\{0\}$ or $\frak h=\frak g$, and the lemma follows.    
\end{proof}

\section{Proof of Theorem \ref{thm.BMS}}\label{sec.proof}
We now discuss the proof of the main theorem.
Let $G$ be a connected semisimple real algebraic group of the form $G=G_1\times\cdots\times G_q$ where each $G_i$ $(1\leq i\leq q)$ is simple and non-compact.
Let $\Ga<G$ be a Zariski dense subgroup, and $\bms$ be as in \eqref{eq.sfm}.
Throughout the section, we will always assume that
$$
\bms(\Ga\ba G)<\infty.
$$
Let $\frak g$ denote the Lie algebra of $G$ and $\frak g_i$ denote the Lie algebra of $G_i$.
We retrieve the notations from Section \ref{sec.B}.
While doing so, we may assume that $\Theta :\frak g\to\frak g$ is chosen to stabilize $\frak g_i$, so that we have a Cartan decomposition $\frak g_i=(\frak k\cap \frak g_i)\oplus (\frak p\cap\frak g_i)$ that is compatible with $\frak g=\frak k\oplus\frak p$.
Set
$$
N_i^\pm=N^\pm\cap G_i,\quad P_i^\pm=P^\pm\cap G_i\quad\text{and}\quad\cal F_i=G_i/P_i^-.
$$

    Let $\pi_i : G\to G_i$ be the projection, $\Ga_i:=\pi_i(\Ga)$, and $\La_{\Ga_i}$ denote the limit set of $\Ga_i$ in $\cal F_i$.
\begin{lem}\label{lem.prod}
    We have $\La_{\Ga}=\prod \La_{\Ga_i}$.
\end{lem}
\begin{proof}
    The lemma follows from a repeated application of Lemma \ref{lem.GG}.
\end{proof}
\begin{remark}\normalfont
    We remark that a major step in the proof of Theorem \ref{thm.Q} in \cite{Quint2}, was to show that $\La_{\Ga_i}=\cal F_i$ for $\Ga_i<G_i$ of 
    rank $\geq 2$ \cite[Prop. 3.1]{Quint2}.
\end{remark}

\subsection*{Coarse Lyapunov weights}
For $\alpha\in\frak a^*$, let
$$
\frak g_\alpha=\{X\in\frak g: \op{ad}_Y(X)=\alpha(Y)X\text{ for all }Y\in\frak a\}
$$
and $\Sigma=\{\alpha\in\frak a^*:\frak g_\alpha\neq\{0\}\}$.
A \textit{coarse Lyapunov weight} is an equivalence class in $\Sigma$ where two elements in $\Sigma$ are equivalent when one of them is a positive multiple of the other.
Denote by $[\alpha]$ the coarse Lyapunov weight containing $\alpha\in\Sigma$, and 
$$
\frak g_{[\alpha]}=\bigoplus_{\beta\in[\alpha]}\frak g_\beta
$$
be the corresponding Lie subalgebra.
For $\alpha\in\Sigma-\{0\}$, set
\begin{align*}
    U_{[\alpha]}=\exp(\frak g_{[\alpha]}).
\end{align*}
Given a finite $A$-invariant ergodic measure $\mu$ on $\Ga\ba G$, let $\tilde U_{[\alpha]}$ denote the smallest Zariski closed $A$-normalized subgroups containing the support of $\mu_x^{U_{[\alpha]}}$ for $\mu$-a.e. $x$.
\begin{lem}\cite[Thm. 9.14]{EL}\label{lem.comm}
Let $\mu$ be a finite $A$-invariant ergodic  measure on $\Ga\ba G$, and $[\alpha]$, $[\beta]$ be coarse Lyapunov weights such that $[\alpha]\neq [\beta]\neq [-\alpha]$.
Then $\mu$ is invariant under the group generated by the commutator $[\tilde U_{[\alpha]},\tilde U_{[\beta]}]$.
\end{lem}
Let $\{[\alpha_1],\cdots,[\alpha_\ell]\}$ be the set of all coarse Lyapunov weights such that $U_{[\al_i]}\subset N^+$.
We may rearrange them so that for each $i$, $[\alpha_i]$ does not meet the convex cones generated by $\{[\al_j]: j=i+1,\cdots,\ell\}$.
Let 
$$
\iota : U_{[\alpha_1]}\times \cdots \times U_{[\alpha_\ell]}\to N^+
$$
be the product map.
Let $\frak n$ be the Lie algebra of $N^+$.
\begin{lem}\label{lem.hololo}
    The following polynomial function has a polynomial inverse:
    \begin{align*}
        \frak g_{[\al_1]}\times\cdots\times \frak g_{[\al_\ell]}&\to \frak n\\
        (X_1,\cdots,X_\ell)&\mapsto \log(e^{X_1}\cdots e^{X_\ell}).
    \end{align*}
\end{lem}
\begin{proof}
    By Baker-Campbell-Hausdorff formula, there exist constants $c_{i_1,\cdots,i_N}$ such that $\log(e^{X_1}\cdots e^{X_\ell})$ can be formally written as
    \begin{equation}\label{eq.BCH}
    X_1+\cdots+X_\ell+\sum_{N=1}^\infty\sum_{1\leq i_1,\cdots,i_N\leq \ell} c_{i_1,\cdots,i_N}[X_{i_1},[X_{i_2},\cdots [X_{i_{N-1}},X_{i_N}]]]    
    \end{equation}
    where for each multi-index $(i_1,\cdots,i_N)$, at least two of the indices are required to be distinct.
    Since $\frak n$ is nilpotent, \eqref{eq.BCH} has only finitely many terms and $\Phi(x_1+\ldots+x_l):=\log(e^{x_1}\ldots e^{x_l})$ is a polynomial map from $\frak n$ to $\frak n$. 

    Let us extend $\Phi$ to the complexification $\frak n^{\mathbb C}:=\frak n\otimes \mathbb C.$ The lemma will follow if we manage to show that the complexified map has a polynomial inverse. The complexified map is given by the same formula as $\Phi$, so it still satisfies the following equivariance property
    \begin{equation}\label{eq-equiv1} \Ad(a)\Phi(x)=\Phi(\Ad(a)x), \text{ for } a\in A, x\in \frak n^\mathbb C.\end{equation}

    Since the derivative of $\Phi$ at zero is identity, we can find an open neighbourhood $U\subset \frak n^\mathbb C$ of $0$ such that $\Phi$ restricted to $U$ is biholomorphic onto the image. On the other hand, for any bounded open set $V\subset n^\mathbb C$, there exists an $a\in A$ such that $\Ad(a)V\subset W$ so \eqref{eq-equiv1} implies that $\Phi$ is biholomorphic on $V$. Taking exhaustive sequence of $V$'s we deduce that $\Phi$ is biholomorphic on $\frak n^\mathbb C$ and, in particular, $\Phi^{-1}$ is an entire function on $\frak n^\mathbb C$. We will now prove it is of polynomial growth. 

    Let $x=x_1+\ldots+x_\ell, $ with $x_i\in \frak g_{[\al_i]}^\mathbb C.$ We fix an auxiliary norm $\|\cdot\|$ on $\frak n^\mathbb C$ defined by $\|x\|=\max_{i=1,\ldots,\ell} \|x_i\|_i,$ where $\|\cdot\|_i$ is any chosen norm on $\frak g_{[\al_i]}^\mathbb C.$     
    Choose $H_0$ in the interior of $\frak a^{+}$ and put $a(t)=e^{tH_0}\in A.$ 
    Choosing a representative such that $[\al_i]\subset\{\al_i,2\al_i\}$, we can choose $0<t\leq c_1\log \|x\|, c_1:=\max_{i=1,\ldots, \ell} (2\alpha_i(H_0))^{-1},$ such that $\|\Ad(a(-t))(x)\|\leq 1.$ Note that 
    $$\Phi^{-1}(x)=\Ad(a(t))(\Phi^{-1}(\Ad(a(-t))(x)).$$ 
    
    Putting $C=\sup_{\|\Phi(z)\|\leq 1} \|z\|$ and $c_2=\max_{i=1,\ldots, \ell} 2\alpha_i(H_0)$, we get 
    $$\|\Phi^{-1}(x)\|\leq C e^{c_1t}\leq C\|x\|^{c_1c_2}.$$
    This establishes the polynomial growth bound for $\Phi^{-1}.$ Any entire holomorphic function of polynomial growth is a polynomial. Indeed, using multivariate Cauchy's formula we can now prove that sufficiently high derivatives of $\Phi^{-1}$ vanish so the Taylor series can have only finitely many non-zero terms.
\end{proof}

The following is then obtained by inductive applications of Proposition \ref{lem.pro}:
\begin{lem}\cite[Thm. 9.8]{EL}\label{lem.str}
    If $\bms(\Ga\ba G)<\infty$, then for $\bms$-a.e.,
    $$
    \bms_x^{N^+}\propto \iota (\bms_x^{U_{[\alpha_1]}}\times\cdots\times \bms_x^{U_{[\alpha_\ell]}}).
    $$
\end{lem}
Using Lemma \ref{lem.str}, one can show that all the leafwise measures are non-trivial. This is one of the most important ingredients of the proof.  
\begin{lem}\label{lem.tr}
For any coarse Lyapunov weight $[\alpha]$, we have $\tilde U_{\pm[\alpha]}=U_{\pm[\alpha]}$. 
In particular, $\bms_x^{U_{\pm[\alpha]}}$ are nontrivial for $\bms$-a.e. $x$.
    
\end{lem}
\begin{proof}
    By contradiction, suppose $\tilde U_{[\alpha]}$ is a proper subgroup of $U_{[\alpha]}$.
    Without loss of generality, $U_{[\alpha]}\subset N^+$.
    Labelling the coarse Lyapunov weights as in Lemma \ref{lem.str}, we have $[\al]=[\al_i]$ for some $i$ and
    \begin{equation}\label{eq.CS2}
    \bms_x^{N^+}\propto \iota (\bms_x^{U_{[\alpha_1]}}\times\cdots\times \bms_x^{U_{[\alpha_\ell]}}).
    \end{equation}
    for $\bms$-a.e $x$.
    Fix such $x\in\Ga\ba G$ and let $g\in G$ be such that $x=[g]$.
    Since $\bms_x^{U_{[\alpha]}}$ is supported on $\tilde U_{[\alpha]}$, by \eqref{eq.CS2}, the topological support of $\bms_x^{N^+}$ is contained in 
    $$
    X:=\iota( U_{[\alpha_1]}\times\cdots\times \tilde U_{[\alpha_i]}\times\cdots\times U_{[\alpha_\ell]}).
    $$
    Since $N^+$ and $U_{[\al]}$'s are nilpotent groups, they are isomorphic to their Lie algebras as algebraic varieties, via the exponential map.  
    By Lemma \ref{lem.hololo}, it follows that $X$ is a proper Zariski closed subset of $N^+$.
    It follows from \eqref{eq.2} that $g^{-1}\La_\Ga$ is contained in the image of $X\subset N^+$,
    under the open embedding $N^+\to G/P^-$. 
    This contradicts the Zariski density of $\La_\Ga$ \cite[Lem. 3.6]{Ben}.
\end{proof} 
 We also record the following two standard lemmas:
\begin{lem}\label{lem.rootsys}
Let $\Phi$ be a simple root system (i.e. root system of a complex simple Lie group) in a Euclidean space $V$. Suppose $W\subset V$ is a subspace of codimension at least $2$. Then, 
$$\Phi\setminus W\subset \{ \alpha+\beta \in \Phi \mid \alpha,\beta\in \Phi\setminus W, 
(\bb R\al+\bb R\beta)\cap W=\{0\}
\}.$$
\end{lem}
\begin{proof}
Let $\kappa\in \Phi\setminus W.$ By \cite[Lemma 3.7]{FG}, the set $\{\eta\in \Phi| \langle\eta, \kappa\rangle\neq 0\}$ spans $V$, so we can find an $\eta$ such that $\langle \eta,\kappa\rangle\neq 0$ and $\kappa+W, \eta+W$ are linearly independent in $V/W.$ The subspace $E:=\mathbb R \kappa+\mathbb R\eta$ intersects $W$ transversally: $E\cap W=\{0\}.$ Consider the intersection $\Phi\cap E$. It is a rank $2$ root system containing a pair of non-orthogonal linearly independent roots $\kappa,\eta$. It follows that $\Phi\cap E$ is of type $A_2, C_2, D_2$ or $G_2$. In each of these root systems one can check by hand that $\kappa=\pm\alpha+\pm\beta$ for some $\alpha,\beta\in \Phi\cap E$. Condition $E\cap W=\{0\}$ implies that 
$(\bb R\al+\bb R\beta)\cap W=\{0\}$.
\end{proof}
\begin{lem}\label{lem.pair}
    Assume that $G$ is a simple Lie group of 
    rank $\geq 2$. Then $G$ is generated by the subgroups $\{[U_{[\alpha]},U_{[\beta]}]: \alpha,\beta\in \Sigma, [\al]\neq[\beta]\neq-[\al]\}$ and $AM$.
\end{lem}
\begin{proof}
    Let $\frak g'$ be the sub-algebra generated by $\{[\frak g_{[\alpha]}, \frak g_{[\beta]}] : \alpha,\beta\in \Sigma,\be\neq-\al\}$, $\frak a$, and $\frak m$, where $\frak m$ denotes the Lie algebra of $M$.
    It suffices to verify that $\frak g'=\frak g$, or equivalently their complexifications satisfy $(\frak g')^{\bb C}=\frak g^{\bb C}$.
    
    Let $\frak h =\frak a+ i\frak b\subset \frak g^{\bb C}$ be a Cartan sub-algebra containing $\frak a$. Let $\Phi$ be the root system of $\frak g^{\bb C}$ relative to $\frak h$. To distinguish them from the roots in $\Sigma$, we will use $\alpha', \beta'$ e.t.c for roots in $\Phi$ and $\alpha, \beta$ e.t.c for roots in $\Sigma$. 
    For $\alpha'\in\Phi$, write
    $$
    \frak g^{\bb C}_{\alpha'}=\{X\in\frak g^{\bb C}: \op{ad}_Y(X)=\alpha'(Y)X\text{ for all }Y\in\frak h\}
    $$
    Let $V:=\Hom(\frak h, \mathbb R)$  and let $W:=\{\xi\in V\mid \xi(\frak a)=0\}.$ Note that the higher rank assumption means that $W$ is of codimension at least $2$. By Lemma \ref{lem.rootsys}, any root $\kappa'\in \Phi$ which restricts to a non-trivial root of $\frak a$ (i.e. $\kappa'\not\in W$) can be written as $\kappa'=\alpha'+\beta'$, with $(\bb R\al'+\bb R\beta')\cap W=\{0\}$.
    Writing $\alpha, \beta$ for the restrictions of $\alpha',\beta'$ to $\frak a$, this in particular implies that $[\al]\neq[\beta]\neq-[\al]$.
    We have 
    $$\frak g^{\bb C}_{\kappa'}=[\frak g^{\bb C}_{\alpha'},\frak g^{\bb C}_{\beta'}]\subset [\frak g_{[\alpha]}, \frak g_{[\beta]}]^{\mathbb C}\subset (\frak g')^{\bb C}.$$
    By \cite[Lemma 3.7]{FG}, $\frak g^{\mathbb C}$ is generated by the root sub-spaces $\frak g^{\bb C}_{\kappa'}, \kappa'\in \Phi\setminus W$ so we can deduce that $(\frak g')^{\bb C}=\frak g^{\bb C}$.
\end{proof}

\begin{prop}\label{lem.rk2}
Assume that $\bms(\Ga\ba G)<\infty$.
Then for each simple factor $G_i$ of 
rank $\geq2$, $\bms$ is $G_i$-invariant.
\end{prop}
\begin{proof}
    By Lemmas \ref{lem.comm} and \ref{lem.tr}, the measure $\bms$ is invariant under the subgroups $\{[U_{[\alpha]},U_{[\beta]}]: \alpha,\beta\in \Sigma, \beta\neq-\al\}$. Since $G_i$ is simple of 
    rank $\geq 2$, Lemma \ref{lem.pair} gives the conclusion.
\end{proof}
At this point we already completed the proof for the simple higher rank case.
To deal with the products of rank one groups, we will once again take advantage of the product structure of the leafwise measures (Lemma \ref{lem.pro}) in the proof of the following:

\begin{prop}\label{lem.d}
    Assume that $\bms(\Ga\ba G)<\infty$ and $\pi_i(\Ga)$ is dense for each $i$.
    Then $\bms$ is Haar.
\end{prop}
\begin{proof}
    Let $\psi$ and $\psi\circ\i$ be the linear forms on $\frak a$ implicitly given in \eqref{eq.sfm}. 
    Identifying $\frak a_i:=\frak a\cap\frak g_i$ with its embedded image in $\frak g$, $\psi$ can be written uniquely as a sum of linear forms $\psi_i$ on $\frak a_i$.
    For any $g\in G$, let $g_i\in G_i$ be its projection to $G_i$.
    Denoting by $\sigma_i : G_i\times \cal F_i\to\frak a_i$ the Iwasawa cocycles for $G_i$ as in \eqref{eq.Iwa}, we have
    $$
    \textstyle\sigma(g,\xi)=\sum_{i=1}^q\sigma_i(g_i,\xi_i)
    $$
    for all $\xi=(\xi_1,\cdots,\xi_q)\in\cal F$.
    Let $\iota : N_1^+\times\cdots N_q^+\to N^+$ denote the product map.
    Since $\bms$ is $A$-invariant, by a repeated application of Lemma \ref{lem.pro},
    \begin{equation}\label{eq.CS4}
    \bms_x^{N^+}\propto \iota (\bms_x^{N_1^+}\times \cdots\times \bms_x^{N_q^+})    
    \end{equation}
    for $\bms$-a.e. $x$.
    Fix such $x\in\Ga\ba G$ and let $g=(g_1,\cdots,g_q)\in G$ be such that $x=[g]$. 
    We then define measures $\nu_i$ on $\cal F_i$ $(1\leq i\leq q)$ by
    \begin{equation}\label{eq.nui}
    d\nu_{i}(g_in_iP_i^-):=e^{\psi_i(\sigma_i(g_in_i,P_i^-))}\,d\bms_x^{N_i^+}(n_i).    
    \end{equation}
    By choosing the normalization of $\bms_x^{N^+}$ as in \eqref{eq.2}, it follows from \eqref{eq.CS4} and \eqref{eq.nui} that
    $$
    \nu_\psi\propto\nu_1\times\cdots\times \nu_q.
    $$
    In particular, $\nu_i$ are finite measures on $\cal F_i$ and for all $\ga\in\Ga$,
    $$
    \ga_1.\nu_1\times \cdots\times\ga_q.\nu_q\propto \ga.\nu_\psi=e^{-\psi_1(\sigma_1(\ga_1^{-1},\cdot))}\nu_1\times\cdots\times e^{-\psi_q(\sigma_q(\ga_q^{-1},\cdot))}\nu_q
    $$
    by \eqref{eq.PS}.
    It follows that there exists a multiplicative character $\chi_i : \Ga\to\bb R^\times$ $(1\leq i\leq q)$ such that for all $\ga\in\Ga$ and $\xi_i\in\cal F_i$,
    $$
    \frac{d(\ga_i.\nu_i)}{d\nu_i}(\xi_i)= \chi_i(\ga)e^{-\psi_i(\sigma_i(\ga_i^{-1},\xi_i))}.
    $$
    Set $\Ga_0:=[\Ga,\Ga]$ and note that each $\chi_i$ is trivial on $\pi_i(\Ga_0)$.
    In particular, $\nu_i$ is $(\pi_i(\Ga_0),\psi_i)$-conformal.
    Since $\pi_i(\Ga_0)$ is dense in $G_i$, it follows that $\nu_i$ is $(G_i,\psi_i)$-conformal.
    This implies that $\nu_i$ is the unique $G_i\cap K$-invariant measure on $\cal F_i$ up to normalization, and $\bms_{x}^{N_i^+}$ is Haar $\bms$-a.e.
    The same argument shows that $\bms_{x}^{N_i^-}$ is Haar $\bms$-a.e.
    Since $N_i^\pm$ and $AM$ generate $G$, by Lemma \ref{lem.inv}, it follows that $\bms$ is Haar.
\end{proof}


We are now ready to give:
\subsection*{Proof of Theorem \ref{thm.BMS}}
The validity of the theorem depends only on the commensurability class of $\Ga$.
Hence, by passing to a covering and a quotient of $G$, we may assume without loss of generality that $G$ is a direct product of its 
simple factors, none of which are compact.
This allows us to freely use the lemmas established in this section.

Let $\pi_i : G\to G_i$ $(i=1,\cdots,q)$ denote the projection of $G$ to each factor.
Since $\pi_i(\Ga)$ is Zariski dense in $G_i$, it is either discrete or dense by Lemma \ref{lem.DD}.
Let $I$ be the set of indices for which 
$\pi_i(\Ga)$ is dense, and $J$ be the complement of $I$.
Let $G_I$ (resp. $G_J$) be the product of $G_i$'s with $i\in I$ (resp. $i\in J$).
Set
$$
\Ga_J:=\prod_{j\in J} \pi_j(\Ga)< G_J,
$$
and $\Ga_I:=\pi_I(\Ga)$ where $\pi_I : G\to G_I$ is the projection.
Applying Lemma \ref{lem.disc} to the product $G=G_I\cdot G_J$, it follows that $\Ga_I$ is discrete, $\Ga$ is commensurable with $\Ga_I\cdot\Ga_J$, and $\bms$ is a finite extension of product of finite $\BMS$ measures on $\Ga_I\ba G_I$ and $\Ga_J\ba G_J$, say $\bms_I$ and $\bms_J$.
Applying Lemma \ref{lem.disc} repeatedly to each $\pi_j(\Ga)$ $(j\in J)$ shows that $\bms_J$ is a product of finite $\BMS$ measures on $\pi_j(\Ga)\ba \tilde\Om_{\pi_j(\Ga)}$.

By Proposition \ref{lem.d}, $\bms_{I}$ is Haar and $\Ga_I$ is a lattice in $G_I$.
By Proposition \ref{lem.rk2}, $\Ga_j$ is a lattice in $G_j$ for all $j\in J$ such that $G_j$ is of rank $\geq 2$, and the remaining indices of $J$ correspond to discrete subgroups of rank one groups with finite $\BMS$ measures.
This complete the proof.
$\qed$


\begin{thebibliography}{10}


\bibitem{Ben} Y. Benoist.
\newblock{\em Proprietes asymptotiques des groupes lineaires.}
\newblock{Geom. Funct. Anal. (1997), 1-47.}




\bibitem{BPS} J. Bochi, R. Potrie, A. Sambarino.
\newblock{\em Anosov representations and dominated splittings.}
\newblock{Jour. Europ. Math. Soc. 21, issue 11 (2019), pp. 3343-3414.}

\bibitem{BCLS}  M. Bridgeman, R. Canary, F. Labourie and A. Sambarino.
\newblock{\em The pressure metric for Anosov representations.}
\newblock{Geom. Funct. Anal. 25 (2015), 1089-1179.}

\bibitem{BLLO} M. Burger,  O. Landesberg, M. Lee, and H. Oh.
\newblock{\em The Hopf-Tsuji-Sullivan dichotomy in higher rank and applications to Anosov subgroups.}
\newblock{Journal of Modern Dynamics, Vol 19 (2023), 269-298.}



\bibitem{DOP} F. Dal'bo, J-P. Otal, and M. Peign\'e,
\newblock{\em S\'eries de Poincar\'e des groupes g\'eom\'etriquement finis.}
\newblock{Israel J. Math. 118 (2000), 109-124.}


\bibitem{EFLO} S. Edwards, M. Fraczyk, M. Lee, and H. Oh. 
\newblock{\em Infinite volume and atoms at the bottom of the spectrum.}
\newblock{Preprint, 2023}

\bibitem{ELO} S. Edwards, M. Lee, and H. Oh. 
\newblock{\em Anosov groups: local mixing, counting, and equidistribution}
\newblock{Preprint, arXiv:2003.14277,}
\newblock{To appear in Geometry \& Topology.}


\bibitem{EO} S. Edwards, and H. Oh.
\newblock{\em Temperedness of $L^2(\Gamma\ba G)$ and positive eigenfunctions in higher rank.}
\newblock{Preprint, arXiv:2202.06203, Under revision for Comm. AMS}

\bibitem{EK} M. Einsiedler, and A. Katok.
\newblock{\em Invariant measures on $G/\Ga$ for split simple Lie groups $G$. (English summary)}
\newblock{Dedicated to the memory of J\"urgen K. Moser.
Comm. Pure Appl. Math. 56 (2003), no. 8, 1184–1221.}

\bibitem{EKL} M. Einsiedler, A. Katok, and E. Lindenstrauss.
\newblock{\em Invariant measures and the set of exceptions to Littlewood's conjecture.}
\newblock{Ann. of Math. (2) 164 (2006), no. 2, 513–560.}

\bibitem{EL} M. Einsiedler, and E. Lindenstrauss.
\newblock{\em Diagonal actions on locally homogeneous spaces.}
\newblock{Homogeneous flows, moduli spaces and arithmetic, 155–241,
Clay Math. Proc., 10, Amer. Math. Soc., Providence, RI, 2010.}



 \bibitem{FockG} V. Fock and A. Goncharov.
\newblock{\em Moduli spaces of local systems and higher Teichumuller theory.}
\newblock{Publ. IHES, Vol 103 (2006), 1-211.}

\bibitem{FG} M. Fraczyk, and T. Gelander.
\newblock{\em Infinite volume and injectivity radius.}
\newblock{Ann. of Math. (2) 197 (2023), no. 1, 389-421.}





\bibitem{GG} F. Gu\'eritaud, O. Guichard, F. Kassel, and A. Wienhard.
\newblock{\em Anosov representations and proper actions.}
\newblock{ Geom. Topol., vol. 21 (2017), 485-584.} 

\bibitem{GW} O. Guichard and A. Wienhard.
\newblock{\em Anosov representations: Domains of discontinuity and applications.}
\newblock{Inventiones Math., Volume 190, Issue 2 (2012), 357-438.}

\bibitem{Hi}
N. J. Hitchin.
\newblock{\em Lie groups and Teichmuller space.}
\newblock{Topology 31 (1992), no. 3, 449-473.}

\bibitem{KLP} M. Kapovich, B. Leeb, and J. Porti. 
\newblock{\em Anosov subgroups: dynamical and geometric
characterizations. }
\newblock{Eur. J. Math. 3 (2017), no. 4, 808-898.}

\bibitem{Ka} F. Kassel.
\newblock{\em Geometric structures and representations of discrete groups.}
\newblock{Proceedings
of the ICM-Rio de Janeiro (2018). Vol. II, 1115-1151.}

\bibitem{KL} B. Kleiner, and B. Leeb.
\newblock{\em Rigidity of invariant convex sets in symmetric spaces. (English summary)}
\newblock{Invent. Math. 163 (2006), no. 3, 657-676.}



\bibitem{La} F. Labourie.
\newblock{\em Anosov flows, surface groups and curves in projective space.}
\newblock{Invent. Math. 165 (2006), no. 1, 51-114.}

\bibitem{Mar} G. A. Margulis.
\newblock{\em Certain measures that are related to Anosov flows.}
\newblock{Funkcional. Anal. i Prilozen. 4 (1970), no. 1, 62-76.}

\bibitem{Mar2} G. A. Margulis.
\newblock{\em Discrete subgroups of semisimple Lie groups.}
\newblock{Ergebnisse der Mathematik und ihrer Grenzgebiete (3) [Results in Mathematics and Related Areas (3)], 17. Springer-Verlag, Berlin, 1991. x+388 pp. ISBN: 3-540-12179-X}




\bibitem{Pat} S. Patterson.
\newblock{\em The limit set of a Fuchsian group.}
\newblock{Acta Math. 136 (1976),  241-273.}

\bibitem{Pes} Y. B. Pesin.
\newblock{\em Characteristic Lyapunov exponents and smooth ergodic theory.}
\newblock{Russ. Math. Surv. 32:4 (1977), 55-114.}

\bibitem{Quint1}  J. F. Quint.
\newblock{\em Mesures de Patterson-Sullivan en rang sup\'erieur.}
\newblock{Geom. Funct. Anal. 12 (2002), no. 4, 776-809.}

\bibitem{Quint3}  J. F. Quint.
\newblock{\em Divergence exponentielle des sous-groupes discrets en rang sup\'erieur.}
\newblock{Comment. Math. Helv. 77 (2002), no. 3, 563-608.}

\bibitem{Quint2}  J. F. Quint.
\newblock{\em Groupes convexes cocompacts en rang sup\'erieur.}
\newblock{Geom. Dedicata 113 (2005), 1–19.}





 \bibitem{Sul} D. Sullivan.
 \newblock{\em The density at infinity of a discrete subgroup of hyperbolic motions.}
\newblock{Publ. IHES (1979), 171-202} 

\bibitem{Win} A. Wienhard.
\newblock{\em An invitation to higher Teichm\"uller theory.}
\newblock{Proceedings of ICM, Rio de Janeiro 2018. Vol. II. 1013-1039.}

\end{thebibliography}
\end{document}